\documentclass[a4paper,20pt]{amsproc}

\usepackage{bbm}
\usepackage{hyperref}
\usepackage{color}
\newtheorem{theorem}{Theorem}[section]

\theoremstyle{definition}
\newtheorem{definition}[theorem]{Definition}

\theoremstyle{remark}
\newtheorem{remark}[theorem]{Remark}

\newtheorem{assumption}[theorem]{Assumption}

\usepackage{amssymb} 
\usepackage{graphicx} 

\newcommand{\R}{\mathbb{R}}

\DeclareMathOperator{\Ima}{Im}
\DeclareMathOperator{\id}{id}

\numberwithin{equation}{section}

\begin{document}

\title{Whitney's and Seeley's type of extensions for maps defined on some Banach spaces}
\author[Victoria Rayskin]{Victoria Rayskin}
\email{victoria.rayskin@tufts.edu}
\thanks{{The research was sponsored by the Army Research Office and was
accomplished under Grant Number W911NF-19-1-0399. The views and conclusions contained in this
document are those of the authors and should not be interpreted as representing the official policies, either
expressed or implied, of the Army Research Office or the U.S. Government. The U.S. Government is
authorized to reproduce and distribute reprints for Government purposes notwithstanding any copyright
notation herein.}
}


\date{\today}

\maketitle
\begin{abstract}
Let $X=C[0,1]$, and $Y$ be an arbitrary Banach space. Consider a collection of open segments $\{V_i \}\subset X$. Suppose the map $f: \cup_i V_i \to Y$ has $q$ bounded  Fr\'echet derivatives ($q=0,1,...,\infty$), and $f$ and all its derivatives have continuous bounded limits at the boundary. Then, subject to some non-intercept condition for the segments $V_i$, the map $f$ can be extended to $F: X\to Y$, so that $F_{|\,\cup_i V_i}=f$ and $F$ has $q$ bounded derivatives. We prove similar Whitney's Extension theorem generalizations for some other Banach spaces. We also prove Seeley Extension theorem for $X=C[0,1].$

These results are related to the problems of function approximation, and manifold learning, which are of central importance to many applied fields.

\end{abstract}

\section{Introduction}
With the development of numerical techniques for large dataset, the ideas of fitting smooth functions to the data becomes increasingly important. 
Recent progress in extension theory has been driven by mathematical work on Whitney's extension problem. See, for example, \cite{F}. 

Usually, the Whitney extension problem is formulated for a function $f: E \to \R^n$, where $E\subset \R^n$. The goal is to decide whether $f$ extends to a $C^m$ function $F$ on $R^n$. If $E$ is finite, then one can efficiently compute such $F$ with the $C^m$ norm having the least possible order of magnitude.

 Interpolation of data by smooth functions is related to the more difficult problem of manifold learning, in which one attempts to pass a smooth surface through the set of data points. In some cases, a better fitting can be achieved, if the original finite dimensional data is transformed to some infinite dimensional  space, where a linear operator can produce a non-planar smooth surface in $\R^n$. For example, with the development of machine learning technique and kernel methods (see, for example, \cite{HSS}), RBF kernel became popular. It can help to approximate data points with non-flat curves. When using RBF kernels, we are solving a linear optimization problem with a kernel operator acting on some infinite dimensional space. For such methods, the existence of a smooth solution to a fitting problem in infinite-dimensional spaces is important.

 In this manuscript, we consider the Whitney extension problem for the functions defined on (infinite dimensional) Banach spaces. Namely, we prove the following results.

Let $X=C[0,1]$, and $Y$ be an arbitrary Banach space. Consider a collection of open segments $\{V_i \}\subset X$ (precisely defined in Section~\ref{section-Whitney}). Suppose the map $f: \cup_i V_i \to Y$ has $q$ bounded  Fr{\'e}chet derivatives ($q=0,1,...,\infty$), and $f$ and all its derivatives have continuous bounded limits at the boundary. Then, subject to some non-intercept condition for the segments $V_i$, the map $f$ can be extended to $F: X\to Y$, so that $F_{|\,\cup_i V_i}=f$ and $F$ has $q$ bounded derivatives. We prove similar results for some other Banach spaces, which generalize the Whitney Extension Theorem for these Banach spaces. We also prove Seeley Extension theorem for $X=C[0,1].$

The main difficulty in the proofs of the extension results on Banach spaces is the absence of smooth bump functions on many Banach spaces, in particular, on $C[0,1]$. Our earlier works (\cite{BR1} and \cite{BR2}) demonstrate how to overcome this obstacle.

In the work~\cite{BR1} we defined $C^q$-blid maps  on Banach spaces (also, see Definition~\ref{def-blid}). These maps  play role similar to bump functions on a broader class of spaces, in particular, on the spaces that do not possess smooth bump functions. In the work~\cite{BR2} we extended the idea of differentiable blid-map to differentiable blid property on linear topological spaces. This technique allows smooth localization on such non-smooth spaces as $C[0,1]$ and differentiable localization on $C^{\infty}(\R)$.

\section{Whitney type of extension result for the spaces possessing smooth blid-maps}\label{section-Whitney}
Let $X= C[0,1]$. Consider the collection $V_i\subset X$ ($i\in I$)  of open segments $V_i=(\phi_i, \psi_i)=\{x\in X\ |\ \phi_i(t)< x(t) < \psi_i(t),\ t\in [0,1] \}$ on the space $X$. Define $E$ as 
$$
E=\cup_{i\in I} V_i
$$

Let $Y$ be a Banach space. Suppose $f:{E}\to Y$  has a continuous and bounded extension to the boundary of $E$. Our goal is to extend $f$ to the entire space $X$ preserving smoothness and boundedness of its derivatives, if $f$ is smooth and has bounded derivatives. For our construction, the map $f$ has to be uniquely defined at the intercepts of the functions $x(t)$, contained in $E$. If $x_i\in \bar{V_i}$, $\ x_j\in \bar{V_j}$ ($i\neq j$), then for any $t_0\in [0,1]$, $x_i(t_0)$ should not be equal to $x_j(t_0)$. Thus, we will make the following
\begin{assumption}[Non-intercept]\label{intersect}
For any pair $i\neq j$, either $\psi_i(t) < \phi_j(t)$, or $\psi_j(t) < \phi_i(t)$. 
\end{assumption}

\begin{theorem}\label{thm1}
Suppose $f:E\to Y$ has $q$ bounded  Fr{\'e}chet derivatives on $E$ ($q=0,1,...,\infty$), and $f$ and all its derivatives have continuous bounded limits at the boundary. Suppose the non-intercept assumption~\ref{intersect} is satisfied. Then, there is an extension $F: X\to Y$, such that $F_{|E}=f$ and $F$ has $q$ bounded derivatives on $X$.
\end{theorem}
\begin{proof}
For each $i\in I$, fix an arbitrary point $z_i \in V_i$. \\
Let ${\bf h}_i \left(r\right): X\to X$ be a map, such ${\bf h}_i \left(r\right)(t) = {h_i} \left(r(t), t \right)$ is a $C^{\infty}$ bump function  $h_i:\R \times [0,1]\to [0,1]$  (for a fixed $t\in[0,1]$, ${h_i} \left(r(t), t \right)=h_i(r)$ is a $C^{\infty}$ bump function on $\R$). Namely,
$$
{\bf h}_i (r)(t)={h_i} \left(r(t), t \right)=
\left\{
\begin{array}{l}
 1 \mbox{ , if  } \phi_i(t)-z_i(t)\leq r(t)\leq \psi_i(t)-z_i(t)\\
0 \mbox{ , if  }  r(t) \leq\phi_i(t)-z_i(t)-\delta_i \mbox{ or } r(t) \geq \psi_i(t)-z_i(t)+\delta_i ,
\end{array}
\right.
$$
where for each $i\in I$ the open set $\left\{x\in X\,:\, \phi_i(t) -\delta_i < x(t) < \psi_i(t) +\delta_i  \right\}$ contains $V_i$, and the positive real numbers $\delta_i$  are  sufficiently small so that $$\cap_{i\in I}      \left\{x\in X\,:\, \phi_i(t) -\delta_i < x(t) < \psi_i(t) +\delta_i  \right\}       = \emptyset.$$
Define
\begin{equation*}
H_i(r)= {\bf h}_i \left(r\right)\cdot r.
\end{equation*}

Obviously $H_i$ are smooth maps possessing bounded derivatives on $X$, having the following properties: 
\begin{itemize}
\item $H_i=\id$ on  $\left[\phi_i-z_i ,\, \psi_i-z_i\right]$
\item $H_i$ are supported on $\left[\phi_i-z_i -\delta_i,\, \psi_i-z_i+\delta_i\right]$
\item $H_i(r)$ are bounded on X, i.e., $\sup_r ||H_i(r)||_X <\infty$.
\end{itemize}
The $H_i$ maps will help us to construct extensions from $V_i$ to $X$. This idea of extension was first introduced in~\cite{BR1} and the maps $H_i$ were called smooth blid-maps. 
\begin{definition}\label{def-blid}
A $C^q$ map $H: X \to X$ is called a $C^q$ blid-map if there exists a neighborhood $U \subset X$, such that $H_{|U}=\id$, and $\sup_x ||H(x)||_X<\infty$. In other words, the maps $H$ is a {\bf b}ounded {\bf l}ocally {\bf id}entical map on $X$.
\end{definition}
Now we can shrink the image of $H$ as in~\cite{BR1} for the map extension construction.
There exists sufficiently small $\epsilon>0$ such that for all $i$ and for all $r\in X$ 
\begin{equation*}
H_i^{\epsilon}(r)= \epsilon H_i\left(\frac{1}{\epsilon }r\right)
\end{equation*}
maps $X$ into $V_i$ translated by $z_i$, i.e, $H_i^{\epsilon}:X\to  \left( \phi_i-z_i,   \psi_i-z_i\right)$.
Then, for every $x\in X$, written as $x=z_i+r$ we can define $F_i(x):X\to Y$ as
\begin{equation*}\label{system}
F_i(x):=f_i(z_i+H_i^{\epsilon}\left(r) \right).
\end{equation*}
The maps $F_i=f$ on $V_i$.\\
Finally, the extension map $F:X\to Y$
\begin{equation*}\label{eq2}
F(x):=\sum_{I} {\bf h}_i\left(x-z_i\right) \cdot  F_i(x)
\end{equation*}
is such that $F_{|E}=f$.
The map $F$ has the same smoothness as $f$ and all the derivatives of $F$ are bounded if the corresponding derivatives of $f$ are bounded. 

\end{proof}

A similar extension can be constructed on any Hilbert space $X$. 
\begin{theorem}\label{thm-Hilbert-space}Let $X$ be an arbitrary Hilbert space, $Y$ be a Banach space and $z_i\in X$, $a_i\in \R$ for $i\in I$.  Suppose the open segments $V_i:= \{ x\in X: ||x-z_i|| <a_i \}$ satisfy the following non-intercept restriction. For any pair $i\neq j$ and for any $x\in E$, if $||x-z_i|| \leq a_i$, then $|| x-z_j|| > a_j$, where $E=\cup_i V_i$. 

Suppose $f:E\to Y$ has $q$ bounded  Fr{\'e}chet derivatives on $E$ ($q=0,1,...,\infty$), and $f$ and all its derivatives have continuous bounded limits at the boundary. Then, there is an extension $F: X\to Y$, such that $F_{|E}=f$ and $F$ has $q$ bounded derivatives on $X$.
\end{theorem}
\begin{proof}
To construct the extension, we define
$$
H_i(r)= h_i(||r||^2)\cdot r,
$$
where $h_i:\R \to [0,1]$ is a regular smooth bump function, such that 
$$
{h_i}(a) =
\left\{
\begin{array}{l}
 1 \mbox{ , if  } |a| \leq a_i\\
0 \mbox{ , if  }  |a| \geq a_i+\delta_i.
\end{array}
\right.
$$
We assume that $\delta_i$ are small, so that $\cap_i \left[a_i-\delta_i, \, a_i+\delta_i\right] =\emptyset$
Then, for sufficiently small $\epsilon$, the norms of $H_i^{\epsilon}(r)=\epsilon H_i(r/\epsilon)$ can be made less than $a_i$ and the following $F_i(x)$ are well-defined:   
\begin{equation}\label{F_i-Hilbert}
F_i(x)=f_i(z_i+H_i^{\epsilon}\left(r) \right).
\end{equation}
The extension $F$ can be defined as
\begin{equation*}\label{F-Hilbert}
F(x)=\sum_{I} { h}_i\left(||x-z_i||^2\right) \cdot  F_i(x).
\end{equation*}
\end{proof}

For the set $E\subset X$ defined as in Theorem~\ref{thm-Hilbert-space}, we can construct a map extension, when the spaces $X$ is $C(M)$.
\begin{theorem}\label{theorem-C-blid-space}
Suppose $M$ is a compact Hausdorff space and $X=C(M)$. 
Let $Y$ be a Banach space and $z_i\in X$, $a_i\in \R$ for $i\in I$.  Suppose the open segments $V_i:= \{ x\in X: ||x-z_i|| <a_i \}$ satisfy the following non-intercept restriction. For any pair $i\neq j$ and for any $x\in E$, if $||x-z_i|| \leq a_i$, then $|| x-z_j|| > a_j$, where $E=\cup_i V_i$. 

Suppose $f:E\to Y$ has $q$ bounded  Fr{\'e}chet derivatives on $E$ ($q=0,1,...,\infty$), and $f$ and all its derivatives have continuous bounded limits at the boundary. Then, there is an extension $F: X\to Y$, such that $F_{|E}=f$ and $F$ has $q$ bounded derivatives on $X$.
\end{theorem}
\begin{proof}
Here the maps $H_i$ can be defined as $H_i(r)=h_i(\sup_{t\in M}|r|)\cdot r$,
where $h_i:\R \to [0,1]$ are the regular smooth bump functions, such that 
$$
{h_i}(a) =
\left\{
\begin{array}{l}
 1 \mbox{ , if  } |a| \leq a_i\\
0 \mbox{ , if  }  |a| \geq a_i+\delta_i ,
\end{array}
\right.
$$
The maps $F_i$ can be defined by Equation~\eqref{F_i-Hilbert}, as in Theorem~\ref{thm-Hilbert-space}. Then, the extension map $F$ is 
\begin{equation*}\label{F-C(M)}
F(x)=\sum_{I} { h}_i\left(\sup_{t\in M}|x-z_i|\right) \cdot  F_i(x).
\end{equation*}
\end{proof}

Consider an arbitrary Banach space possessing a smooth blid map. Assume that the set $E$ consists of a single open segment. Then the map extension theorem can be formulated as below.
\begin{theorem}\label{remark-blid-space} 
Let $X$ be a Banach space possessing smooth blid map $H: X\to X$. Let $Y$ be a Banach space and $z\in X$, $a\in \R$ and $E= \{ x\in X: ||x-z|| <a \}$. 

Suppose $f:E\to Y$ has $q$ bounded  Fr{\'e}chet derivatives on $E$ ($q=0,1,...,\infty$), and $f$ and all its derivatives have continuous bounded limits at the boundary. Then, there is an extension $F: X\to Y$, such that $F_{|E}=f$ and $F$ has $q$ bounded derivatives on $X$.
\end{theorem}
\begin{proof}
\begin{equation*}\label{F-blid-space}
F(x)=f(z+\epsilon H\left(r/{\epsilon}) \right),
\end{equation*}
where $\epsilon$ is small enough to guaranty that $z+\epsilon H\left(r/{\epsilon}\right)\subset E$.
\end{proof}
\begin{remark}\label{remark-semi-finite}
In the theorem~\ref{thm1}, for any $i\in I$ either $\phi_i$ or $\psi_i$ can be replaced with $\infty$, as long as the Assumption~\ref{intersect} is satisfied. I.e., the finite open segment $(\phi_i, \psi_i)=\{x\in X\ |\ \phi_i(t)< x(t) < \psi_i(t),\ t\in [0,1] \}$ can be replaced with one of the semi-finite sets $(-\infty, \psi_i)=\{x\in X\ |\ x(t) < \psi_i(t),\ t\in [0,1] \}$ or $(\phi_i, \infty)=\{x\in X\ |\ \phi_i(t)< x(t),\ t\in [0,1] \}$.
\end{remark}

\section{Seeley type of extension for the spaces $X=C[0,1]$.}
With the help of the Theorem~\ref{thm1} and the Remark~\ref{remark-semi-finite}, we can generalize Seeley's Extension Theorem to the case, where the domain of the map is a Banach space $C[0,1]$.

Namely, let the space $X=U \bigoplus W$, where $W=\ker(P)\neq \left\{0\right\}$ and $U=\Ima (P)$ for some orthogonal projector $P$ on the space $X$.
\\
If $W \neq X$, define
$$
E:=\left\{ (u,w)\in X \, |\,  u\in U, w\in W, \ w(t)<0 \right\},
$$
and if $W=X$, define
$$
E:=\left\{ x\in X \, |\,  x(t)<0 \right\}.
$$

\begin{theorem}\label{thm-seeley}
Let  the half-space $E\subset X$ be as above. Suppose $f:E\to Y$ smoothly extends to the boundary of $E$ and has $q$ bounded derivatives, which have bounded limit at the boundary of $E$ ($q=0,1,...$). Then, there is an extension $F: X\to Y$, such that $F_{|E}=f$ and $F$ has $q$ bounded derivatives.
\end{theorem}
\begin{proof}
If $W\neq X$
$$
F(u,w) = f(u, z+H^{\epsilon}(r)),
$$
otherwise
$$
F(w) = f(z+H^{\epsilon}(r)).
$$
Here $z+r=w$, $\phi =\infty$, $\psi\equiv 0$ and $H^{\epsilon}(r)$ is constructed as in the Theorem~\ref{thm1}.
\end{proof}

\bibliographystyle{amsalpha}

\end{document}